\documentclass[12pt]{amsart}

\usepackage{amssymb,mathrsfs,bm}

\usepackage{enumerate}

\usepackage[centering]{geometry}
\usepackage{hyperref}
\usepackage[square, comma, sort&compress, numbers]{natbib}
\theoremstyle{plain}
\newtheorem{theorem}{Theorem}[section]
\newtheorem{lemma}[theorem]{Lemma}
\newtheorem{corollary}[theorem]{Corollary}

\theoremstyle{definition}
\newtheorem{definition}[theorem]{Definition}

\newtheorem{remark}[theorem]{Remark}

\numberwithin{equation}{section}
\newtheoremstyle{citing}
{3pt}
{3pt}
{\itshape}
{}
{\bfseries}
{.}
{.5em}
{\thmnote{#3}}
\theoremstyle{citing}

\begin{document}
\title{ Modified shrinking target problem for Matrix Transformations of Tori
}
\author{Na Yuan}

\address{School of Mathematics,
Guangdong University of Education,  Xingang Middle Road, 351, Guangzhou, P.~R.\ China}

\email{yna@gdei.edu.cn}

\author{ShuaiLing Wang$^{*}$}

\address{Department of Mathematics, South China University of Technology,
	Guangzhou, 510641, P.~R.\ China}

\email{scslwang@scut.edu.cn}



\thanks{* Corresponding author}

\date{}
\maketitle

\begin{center}
	\begin{minipage}{120mm}
		{\small {\bf Abstract.}
  We  calculate  the Hausdorff dimension of the fractal set
\begin{equation*}
 \Big\{\mathtt{x}\in \mathbb{T}^d: \prod_{1\leq i\leq d}|T_{\beta_i}^n(x_i)-x_i| < \psi(n) \text{\ for  \ infinitely \ many }  n\in \mathbb{N}\Big\},
\end{equation*}
where the  $T_{\beta_i}$ is the standard 
$\beta_i$-transformation with $\beta_i>1$, $\psi$  is  a  positive function on $\mathbb{N}$ and   $|\cdot|$ is the usual metric on the torus $\mathbb{T}$.
 Moreover,
 we  investigate  a modified version of  the shrinking target problem,  which unifies the shrinking target problems and quantitative
     recurrence properties for matrix transformations
of tori. 
Let $T$ be a $d\times d$ non-singular matrix with real coefficients. Then, $T$ determines a self-map of the $d$-dimensional torus $\mathbb{T}^d:=\mathbb{R}^d / \mathbb{Z}^d$. 
 For any  $1\leq i \leq d$, let $\psi_i$ be a  positive function on $\mathbb{N}$ and  $\Psi(n):=(\psi_1(n),\dots, \psi_d(n))$ with  $n\in \mathbb{N}$.
We obtain the Hausdorff dimension of the fractal set
 \begin{equation*}
 \big\{\mathtt{x}\in \mathbb{T}^d: T^n(x)\in L(f_n(\mathtt{x}), \Psi(n)) \text{\ for \ infinitely \ many } n\in \mathbb{N}\big\},
\end{equation*}
where $L(f_n(\mathtt{x}, \Psi(n)))$  is a hyperrectangle and $\{f_n\}_{n\geq 1}$ is a sequence of Lipschitz vector-valued functions on $\mathbb{T}^d$ with a uniform   Lipschitz constant.

		}
	\end{minipage}
\end{center}

\section{Introduction}
Let $(X,d)$ be a metric space,  and  let $T:X\to X$  be a transformation.
 If $\mu$ is a $T$-invariant
Borel probability measure on $X$ and $T$ is ergodic with respect to the measure $\mu$,  Birkhoff's  ergodic theorem implies 
that for any  ball  $B\subset X$ of  positive measure, the subset 
$
    S=\{x\in X: T^n x\in B \text{\ for \  i.m.\ }  n\in \mathbb{N}\}$ 
has full $\mu$-measure.
Here and throughout the paper, `i.m.' stands for `infinitely many'.
This means that the trajectories of almost all points  will enter the ball $B$ infinitely often.
In general, one may wonder what will happen if $B$ shrinks with time.
Let $\psi$ be a function on $\mathbb{N}$ and $z\in X$,
  the investigation of the size in terms of measure and dimension of 
  the set 
\begin{equation*}
    S(\psi,z)=\{x\in X: T^n x\in B(z,\psi(n)) \text{\ for \  i.m.\ }  n\in \mathbb{N}\},
\end{equation*}
is called shrinking target problems by  Hill and Velani \cite{Hill1995}, where 
 $B(z,\psi(n))\subset X$ is a ball of center $z$ and radius $\psi(n)$.
On the other hand, motivated by Poincar\'{e}'s recurrence theorem in dynamical systems,  one is also interested in the quantitative recurrence set(see\cite{Barreira2001,Boshernitzan1993})
\begin{equation*}
    R(\psi):=\{x\in X: T^n x\in B(x,\psi(n)) \text{\ for \  i.m.\ }  n\in \mathbb{N}\}.
\end{equation*}

The  measures and dimensions  of the  fractal sets $S(\psi,z)$ and $R(\psi)$  have already been extensively investigated  in many dynamical systems. For the measure aspect,
 the reader is referred to 
\cite{Allen,Baker2019,Chang2018,Chernov2001,Heyubin, Hill2002,Hussain2021,Kirsebom2022,Kleinbock2022,li2022} and
references therein.
With regards to the dimension aspect,
let us cite some but far from all concrete cases.
The  dimension  aspect has been studied for
expanding rational maps of Julia sets \cite{Hill1995,Hill1997},
matrix transformations of tori \cite{Hill1999,li2022},    $\beta$-transformations \cite{HUSSAIN2017,Shen2013ShrinkingTP,Tan2011},  conformal iterated
function systems \cite{Reeve,Seuret2013,Li2014}.
For the more results, the reader is referred to 
\cite{persson_rams_2017,Barany2018ShrinkingTO,Barany2021,WUyuliang} and
references therein.

It should be observed that in many systems, the fractal dimensional formulae for the two sets $S(\psi)$ and $R(\psi)$ are the same.
In \cite{Shen2023,Wang2018,wuyufeng,Yuan2023}, 
the dimensions of $S(\psi)$ and $R(\psi)$ were unified in some dynamical systems  by using a Lipschitz function.
We are interested in finding out if they can be unified when $T$ is a matrix transformation on the torus $\mathbb{T}^d$.

We start by laying out some necessary definitions and notations.
Let $\mathbb{T}^d:=\mathbb{R}^d / \mathbb{Z}^d$ be a  $d$-dimensional torus
endowed with the  usual metric in $\mathbb{T}^d$ induced by the Euclidean metric.
Let $T$ be a $d\times d$ non-singular matrix with real coefficients. Then, $T$ determines a self-map of the $d$-dimensional torus $\mathbb{T}^d$, namely, for any  $\mathtt{x}\in \mathbb{T}^d$, $T:\mathtt{x} \to T(\mathtt{x}) \ (\bmod 1)$. 
In what follows,  $T$ will denote both the matrix and the transformation.  Denote by $T^n$  the $n$-th iteration of the transformation $T$ with  $n\in \mathbb{N}$.  
Let $\mathbb{R}^{+}$ be the set of positive numbers.

We explore two special sets that are relevant to the classical theory of Diophantine approximation.
  For any  $1\leq i \leq d$, let $\psi_i$ be a  positive function on $\mathbb{N}$ and  $\Psi(n):=(\psi_1(n),\dots, \psi_d(n))$ with  $n\in \mathbb{N}$.
For any $\mathtt{y}=(y_1,y_2,\dots, y_d)\in\mathbb{T}^d$, let 
\begin{equation*}
    L(\mathtt{y}, \Psi(n)):=\Big\{\mathtt{x}\in \mathbb{T}^d:|x_i-y_i| < \psi_i(n) \text{\ for any\ } 1\leq i\leq d\Big\},
\end{equation*}
and
\begin{equation*}
    P(\mathtt{y}, \psi(n)):=\Big\{\mathtt{x}\in \mathbb{T}^d: \prod_{1\leq i\leq d}|x_i-y_i| < \psi(n) \Big\},
\end{equation*}
where $|\cdot|$ is  the usual metric on  $\mathbb{T}$.
Then  $L(\mathtt{y}, \psi(n))$ is a  hyperrectangle and $  P(\mathtt{y}, \psi(n))$ is a hyperboloid. Define 
\begin{equation*}
    H_d(T,\psi)=\big\{\mathtt{x}\in \mathbb{T}^d: T^n(\mathtt{x})\in P(\mathtt{x}, \psi(n)) \text{\ for \  i.m.\ }  n\in \mathbb{N}\big\}.
\end{equation*}
The set $ H_d(T,\psi)$  is intimately related to sets studied within  the multiplicative theory of Diophantine approximation.

Let $\{f_n\}_{n\geq 1}$ be  a sequence of vector-valued functions on $\mathbb{T}^d$ with a  uniform  Lipsitz constant, i.e., 
 there exists $c\in \mathbb{R}^{+}$ such that for any  $n\in \mathbb{N}$ and
$\mathtt{x},\mathtt{y}\in \mathbb{T}^{d}$,
\begin{equation}\label{fnlip}
    |f_n(\mathtt{x})-f_n(\mathtt{y})|\leq c|\mathtt{x}-\mathtt{y}|,
\end{equation}
where we also use  $|\cdot|$ for the usual metric on $\mathbb{T}^d$ induced by the Euclidean
metric.
 Then, \eqref{fnlip} means that 
for any $n\in \mathbb{N}$ and  $x,y\in \mathbb{T}$, 
\begin{equation}\label{lipschitz function}
  |f^{(i)}_n(x)-f^{(i)}_n(y)|\leq c |x-y|  
\end{equation}
where $f^{(i)}_n$ is the function $f_n$ restricted to 
the direction of the $i$-th axis.

Define  the  modified shrinking target set associated to $\Psi$ and $\{f_n\}_{n\geq 1}$ by
\begin{equation*}
    W(T,\Psi, \{f_n\})=\big\{\mathtt{x}\in \mathbb{T}^d: T^n(\mathtt{x})\in L(f_n(\mathtt{x}), \Psi(n)) \text{\ for \  i.m.\ }  n\in \mathbb{N}\big\}.
\end{equation*}
If  $\psi_1=\psi_2=\dots=\psi_d=\psi$,  we write 
\begin{equation*} 
 W(T,\psi,\{f_n\})=\big\{\mathtt{x}\in \mathbb{T}^d: T^n(\mathtt{x})\in L(f_n(\mathtt{x}), \psi(n)) \text{\ for \  i.m.\ }  n\in \mathbb{N}\big\}. 
\end{equation*}
Note that if $f_n=y$  for all $n\in \mathbb{N}$,
where $\mathtt{y}$ is a fixed point in $\mathbb{T}^d$, then the fractal  set $W(T,\psi,\{f_n\})$
is an analogue of the 
 sets explored by the classical simultaneous theory of Diophantine approximation.

 Denote by $\dim_{\mathrm{H}}$ the Hausdorff dimension.
 Let
 \begin{equation*}
     \tau:=\liminf_{n\to \infty}\frac{-\log \psi(n)}{n}.
 \end{equation*}
 
 Our results are listed below.
  \begin{theorem}\label{hypercube}
Let $T$ be a real, non-singular matrix transformation of the torus $\mathbb{T}^d$.
Suppose that $T$ is diagonal and all eigenvalues $\beta_1$, $\beta_2$, $\dots$, $\beta_d$ are of modulus strictly larger than $1$. Assume that $1<\beta_1\leq \beta_2\leq \dots\leq \beta_d$. 
Then
\begin{equation*}
     \dim_{\mathrm{H}}H_d(T,\psi)=d-1+\frac{\log \beta_d}{\tau+\log \beta_d}.
\end{equation*}
\end{theorem}
Let $\mathcal{C}(\Psi)$ be the set of  accumulation  points  $\mathbf{t}=(t_1,\dots,t_d)$ of  the sequence 
 \[ \Big\{ \frac{-\log \psi_1(n)}{n},\dots, \frac{-\log \psi_d(n)}{n}\Big\}_{n\geq 1}.\]
The following Theorem  \ref{main theorem1} gives the value of $\dim_{\mathrm{H}}W(T,\Psi, \{f_n\})$ when 
$T$ is  a real, non-singular diagonal matrix transformation of the torus $\mathbb{T}^d$.
\begin{theorem}\label{main theorem1}
 Let $T$ be a real, non-singular matrix transformation of the torus $\mathbb{T}^d$. Suppose that $T$ is diagonal with all eigenvalues $\beta_i>1$, $i\geq 1$ and  $1<\beta_1\leq \beta_2\leq \dots\leq \beta_d$.  
 Let  $\{f_n\}_{n\geq 1}$ be  a sequence of vector-valued functions on $\mathbb{T}^d$ with a  uniform  Lipsitz constant and $\Psi=(\psi_1,\psi_2,\dots,\psi_d)$, where $\psi_i:\mathbb{N}\to \mathbb{R}^{+}$.
 If $ \mathcal{C}(\Psi)$ is bounded, then 
\begin{equation*}
    \dim_{\mathrm{H}} W(T,\Psi, \{f_n\})=\sup_{\mathbf{t}\in \mathcal{C}(\Psi)}\min_{1\leq i\leq d} \{\lambda_i(\mathbf{t})\},
\end{equation*}
where  
\begin{equation*}
    \lambda_i(\mathbf{t}):=\sum_{j\in \mathcal{Q}^1_i}1+\sum_{j\in \mathcal{Q}^2_i}(1-\frac{ t_j}{\log \beta_i +t_i})+\sum_{j\in \mathcal{Q}^3_i}\frac{\log \beta_j}{\log \beta_i+t_i}
\end{equation*}
and 
\begin{equation*}
    \mathcal{Q}^1_i:=\{1\leq j\leq d:  \log \beta_j >\log \beta_i+t_i\}, \ \mathcal{Q}^2_i:=\{1\leq j\leq d: \log \beta_j+t_j\leq \log \beta_i+t_i \},
\end{equation*}
\[  \mathcal{Q}^3_i:=\{1,2,\dots,d\}\backslash (  \mathcal{Q}^1_i\cup   \mathcal{Q}^2_i).\]
\end{theorem}
\begin{remark}\label{remak1}
Note that the Hausdorff dimension of the set $W(T,\Psi, \{f_n\})$ is independent of the sequence $\{f_n\}_{n\geq 1}$. 
\end{remark}

If for each $1\leq i, j\leq d$,  $\psi_i=\psi_j=\psi$ in Theorem \ref{main theorem1},  we can obtain  the following Corollary \ref{main_corollary}.
 \begin{corollary}\label{main_corollary}
 Let  $\{f_n\}_{n\geq 1}$ be  a sequence of vector-valued functions on $\mathbb{T}^d$ with a  uniform  Lipsitz constant.
 Let $T$ be a real, non-singular matrix transformation of the torus $\mathbb{T}^d$. Suppose that $T$ is diagonal with all eigenvalues $\beta_i>1$, $i\geq 1$. Assume that $1<\beta_1\leq \beta_2\leq \dots\leq \beta_d$, then 
\begin{equation*}
    \dim_{\mathrm{H}}  W(T,\psi, \{f_n\})=\min_{1\leq i\leq d} \Bigg\{\frac{i\log \beta_i-\sum_{j:\beta_j>\beta_i e^{\tau}}(\log \beta_j-\log \beta_i -\tau)+\sum_{j>i}\beta_j}{\tau+\log \beta_i}\Bigg\}.
\end{equation*}

 \end{corollary}
 \begin{remark}
In particular, the work in \cite{Hill1999} concerned about the Hausdorff dimension of $W(T,\Psi, \{f_n\})$,  where $f_n$ is a constant function for each $n\geq 1$. 
However,
when $T$ satisfies the conditions in Corollary \ref{main_corollary} and $f_n=z_n$ with $n\geq 1$,
our result of determining $\dim_{\rm{H}}W(T,\psi,\{f_n\})$ is new. 
\end{remark}

The following Theorem \ref{diagonalizable} extends 
Corollary \ref{main_corollary} from diagonal matrix to the matrix which can be diagonalizable over $\mathbb{Z}$.
 \begin{theorem}\label{diagonalizable}
Let $T$ be an integer, non-singular matrix transformation of the torus $\mathbb{T}^d$. 
Let  $\{f_n\}_{n\geq 1}$ be  a sequence of vector-valued functions on $\mathbb{T}^d$ with a  uniform  Lipsitz constant  and  $\psi:\mathbb{N}\to \mathbb{R}^{+}$.
Suppose that $T$ is diagonalizable over $\mathbb{Z}$ with all eigenvalues $\beta_i>1$, $i\geq 1$. Assume that $1<\beta_1\leq \beta_2\leq \dots\leq \beta_d$. Then 
\begin{equation*}
    \dim_{\mathrm{H}} W(T,\psi, \{f_n\})=\min_{1\leq i\leq d} \Bigg\{\frac{i\log \beta_i-\sum_{j:\beta_j>\beta_i e^{\tau}}(\log \beta_j-\log \beta_i -\tau)+\sum_{j>i}\beta_j}{\tau+\log \beta_i}\Bigg\}.
\end{equation*}
 \end{theorem}

Our paper is organized as follows.
Section \ref{preliminaries} begins with some preparations on $\beta$-transformation and some useful Lemma.
Section \ref{last} provides the proof of Theorem  \ref{hypercube}.
 The proof of Theorem \ref{main theorem1} and Corollary \ref{main_corollary}   can be found in Section \ref{main section}.
More precisely,  we establish the upper bound of $\dim_{\mathrm{H}} W(T,\Psi, \{f_n\})$  in Section \ref{upper bound}, and obtain the lower bound of  $\dim_{\mathrm{H}} W(T,\Psi, \{f_n\})$  in Section \ref{lower}. 
Section \ref{corollary} 
gives the proof of    Corollary \ref{main_corollary}.
 The proof of Theorem \ref{diagonalizable} is presented in Sections \ref{diagona}.
 
\section{Preliminaries}\label{preliminaries}

Let $\beta>1$. Define the $\beta$-transformation $T_\beta:[0, 1)\rightarrow[0, 1)$ by
\begin{equation*}
T_\beta (x)=\beta x-\lfloor\beta x\rfloor,
\end{equation*}
where $\lfloor \zeta\rfloor$ stands for the largest integer no more than $\zeta$. The system $([0,1),T_{\beta})$ is called $\beta$-dynamical system. Then every real number  $x\in[0,1)$ can be expressed uniquely as a finite or an infinite series
\begin{equation}\label{11}
x=\frac{\epsilon_1(x, \beta)}{\beta}+\cdots +\frac{\epsilon_n(x, \beta)}{\beta^n}+\cdots,
\end{equation}
where, for $ n\geq 1$, $ \epsilon_{n}(x, \beta)=\lfloor \beta T_{\beta}^{n-1}x \rfloor $ is called  $n$-th digit of $ x $  (with respect to base $ \beta $). Then formula (\ref{11}) or the digit sequence
$\varepsilon(x,\beta):=(\epsilon_{1}(x,\beta), \epsilon_{2}(x,\beta),\dots )$
is called the $\beta$-expansion of $x$. Sometimes we rewrite (\ref{11}) as
\begin{equation*}
x=(\epsilon_1(x, \beta), \dots, \epsilon_n(x, \beta), \dots).
\end{equation*}
It is clear that, for $n\geq 1$,  the  $n$-th digit $ \epsilon_{n}(x,\beta)\in\mathcal{A}_\beta=\{0, 1, \dots,  \lceil\beta-1 \rceil\}$, where $\lceil\beta-1 \rceil=\min\{j\in \mathbb{N}:j\geq \beta-1\}$. While, not all sequence $\omega\in \mathcal{A}^\mathbb{N}_\beta$ would be a $\beta$-expansion of some $x\in[0,1]$. We call a finite or an infinite sequence $(\epsilon_1, \epsilon_2,\dots )$ admissible if there  exists an $x\in [0,1)$ such that the $\beta$-expansion of $x$ begins with $(\epsilon_1, \epsilon_2,\dots )$.
For $n\geq 1$, let $\Sigma_{\beta}^{n}$ be the set of all admissible sequences of length $n$.

The following widely recognized result from R\'{e}nyi.
\begin{lemma}[\cite{Rnyi1957RepresentationsFR}]\label{cylinder-number}
Let $\beta>1$.  Then for any $n\geq 1$,
\begin{equation*}
\beta^n\leq \#\Sigma_{\beta}^n\leq \frac{\beta^{n+1}}{\beta-1},
\end{equation*}
where $\#$ denotes the cardinality of a finite set.
\end{lemma}
\begin{definition}
For any  $(\varepsilon_{1},\varepsilon_{2},\dots,\varepsilon_{n})\in \Sigma_{\beta}^{n}$, the set 
\begin{equation*}
I_{n,\beta}(\varepsilon_{1},\varepsilon_{2},\dots,\varepsilon_{n})=\{ x\in [0,1): \varepsilon_{i}(x,\beta)=\varepsilon_{i}, 1\leq i\leq n\}
\end{equation*}
is called   a cylinder of order $n$ (with respect to base $\beta$). 
\end{definition}
\begin{remark}\label{cylinder-length}
The set $ I_{n,\beta}(\varepsilon_{1},\varepsilon_{2},\dots,\varepsilon_{n})$ is a left-closed and right-open interval. 
Moreover, the length of $ I_{n,\beta}(\varepsilon_{1},\varepsilon_{2},\dots,\varepsilon_{n})$ satisfies $| I_{n,\beta}(\varepsilon_{1},\varepsilon_{2},\dots,\varepsilon_{n})| \leq \frac{1}{\beta^{n}}$, where the $|\cdot|$ denotes the length of a cylinder.
\end{remark}
\begin{definition}\label{definition full}
Let $\mathtt{w}\in \Sigma^n_{\beta}$.
Then $I_{n,\beta}(\mathtt{w})$ is called a {\it full cylinder of order $n$} if  $|I_{n,\beta}(\mathtt{w})|=\frac{1}{\beta^n}$.
\end{definition}
The following property of full cylinders plays an important role in the proofs of 
Theorem \ref{main theorem1}.
\begin{lemma}{\rm(\cite[Theorem 1.2]{Bugeaud2014DistributionOF})}\label{wangcylinder}
For each $n\geq 1$, there exists at least one full cylinder among every $n+1$ consecutive cylinders of order $n$.
\end{lemma}
We can get the following lemma by  similar idea as in  the proof of \cite[Lemma 4]{wuyufeng}.
\begin{lemma}{\rm(\cite[lemma 4]{wuyufeng})}\label{wu}
Let $\{g_n\}_{n\geq 1}$ be  a sequence of functions on $\mathbb{T}$ with a  uniform  Lipsitz constant. 
 Then for any full cylinder $I_n(\mathtt{w})$
and for  any $\varepsilon>0$,  there exists a point $x_{n,\mathtt{w}}\in I_n(\mathtt{w})$ such that 
\begin{equation*}
    \big|T^n_{\beta}x_{n,\mathtt{w}}-g_n(x_{n,\mathtt{w}})\big|<\varepsilon.
\end{equation*}
\end{lemma}
The following lemma shows that the Hausdorff dimension of a set is invariant under bi-Lipschitz maps.
\begin{lemma}[\cite{Falconer1990}]\label{lipschitz}
Let $X$ be a subset of $\mathbb{R}^d$ and $g$ be a bi-Lipschitz map, i.e., there exist $0<c_1\leq c_2<\infty$ such that  for any $x,y \in X$,
\begin{equation*}
    c_1|x-y|\leq |g(x)-g(y)|\leq c_2|x-y|.
\end{equation*}
Then $\dim_{\mathrm{H}}g(X)=\dim_{\mathrm{H}}X$.
\end{lemma}
Now we will introduce the Mass Transference Principle for rectangles that is useful to get  the lower bound of $\dim_{\mathrm{H}}W(T,\Psi,\{f_n\})$ in   Theorem \ref{main theorem1}.
For our purpose, we only need the following special case of \cite[Theorem 3.4]{Wang-Wu}.
\begin{lemma} {\rm(\cite[Theorem 3.4]{Wang-Wu})}\label{masstranference}
Let $\mathbb{T}^d$ be  $d$-dimensional torus and $\mathcal{L}^d$ be  the $d$-dimensional Lebesgue measure on $\mathbb{T}^d$. For any $1\leq i\leq d$  and $n\geq 1$,  let $x_{i,n}\in \mathbb{T}$, $ r_{i,n}>0$ and  $B(x_{i,n}, r_{i,n})$ be a ball. Suppose that $ r_{i,n} \to 0$ as $n\to \infty$ for each $i\geq 1$. 
If 
\begin{equation}\label{full_measure}
    \mathcal{L}^d(\limsup_{n\to \infty} \prod_{i=1}^d  B(x_{i,n}, r_{i,n}^{a_i})  ) )=\mathcal{L}^d(\mathbb{T}^d),
\end{equation}
where $a_i \geq 0$ for all $1\leq i\leq d$.
Then  for any choice of $(u_i,v_i)$ with $u_i, v_i\geq 0$ and $\frac{u_i}{u_i+v_i}=a_i$, $1\leq i \leq d$ (if $u_i=v_i=0$, then $a_i=0$),  
\begin{equation*}
    \dim_{\mathrm{H}}( \limsup_{n\to \infty}\prod_{i=1}^d B(x_{i,n}, r_{i,n}))\geq \min_{q\in \mathcal{A}}s(q),
\end{equation*}
where 
$ \mathcal{A}=\{u_i+v_i, 1\leq i \leq d \},$
 and for each $q\in \mathcal{A}$, 
 \begin{equation*}
     s(q):=\sum_{i\in \mathcal{Q}^1(q) }1 +\sum_{i\in \mathcal{Q}^2(q) }(1-\frac{v_i}{q})+\sum_{i\in \mathcal{Q}^3(q)}\frac{u_i}{q}
 \end{equation*}
 the sets $\mathcal{Q}^1(q)$,$\mathcal{Q}^2(q)$, $\mathcal{Q}^3(q)$ form a partition of $\{1,\dots,d\}$ defined as
 \begin{equation*}
         \mathcal{Q}^1(q)=\{1\leq i \leq d: u_i\geq q\}, \hspace{4em}  \mathcal{Q}^2(q)=\{1\leq i \leq d: u_i+v_i\leq q\}\backslash \mathcal{Q}^1(q)
 \end{equation*}
 and
 \[\ \mathcal{Q}^3(q)=\{1\leq i \leq d\}\backslash ( \mathcal{Q}^1(q)\cup  \mathcal{Q}^2(q)).\]
 \end{lemma}
 \section{Proof of the Theorem \ref{hypercube}}\label{last}
We begin by presenting some lemma that we will use to prove Theorem \ref{hypercube}.  
\begin{lemma}\label{big}
Let $\beta>1$  and $\mathtt{w}\in \Sigma^{n}_{\beta}$ with $n\geq 1$. For any $n>\log_{\beta} 2$ and  any 
$0\leq \delta_1< \delta_2 $, the set 
\begin{equation*}
   E_{n,\beta}(\mathtt{w},\delta_1,\delta_2):= \big\{x\in I_{n,\beta}(\mathtt{w}):\delta_1\leq |T_{\beta}^n x-x|\leq \delta_2 \big\}
\end{equation*}
can be covered by $4$ intervals of length  $(\delta_2-\delta_1)\beta^{-n}$.
\end{lemma}
\begin{proof}
Note that inside the cylinder $ I_{n,\beta}(\mathtt{w})$,
\[T_{\beta}^{n}x-x=\beta^n\Big(x-\frac{w_1}{\beta}-\cdots-\frac{w_2}{\beta^n}\Big)-x,\]
is a linear map with slope $\beta^n-1$, where $x=(w_1,w_2,\dots,w_n,\dots)$. Thus, for any $n>\log_{\beta} 2$, $E_{n,\beta}(\mathtt{w},\delta_1,\delta_2)$
consists of at most two intervals of length 
\[ \frac{\delta_2-\delta_1}{\beta^n-1}\leq 2(\delta_2-\delta_1)\beta^{-n}.\]
\end{proof}
Let $\delta>0$ and define 
\[ H_{d}(\delta)= \Big\{\mathbf{y}=(y_1,\cdots,y_d)\in [0,1]^{d}: \prod_{i=1}^{d}y_i\leq \delta \Big\}.\]
The following  result  is important in the proofs of Theorem \ref{hypercube}.
\begin{lemma}\rm{(\cite[Lemma 1]{Bovey1978})}\label{big1}
    Let $d\in \mathbb{N}$ and $\delta$ be a sufficiently small positive number.
Let $n$ be a sufficiently large integer. Then, for any $s\in (d-1,d)$, the set 
$H_{\delta}$ has a covering $\mathcal{B}$ by $d$-dimensional cubes $B$ such that 
\[\sum_{B\in \mathcal{B}}|B|^{s}\ll \delta^{s-d+1},\]
where $|B|$  is the length of a side of $B$ and 
$\ll$	implies an inequality with a factor independent of $\delta$.
\end{lemma}
For a cube $B\subset \mathbb{T}^{d}$, write it as 
\[ B=[a_{B}, b_{B}]^d, \ \text{with } b_{B}-a_{B}=|B|. \]
Define 
\[ E_{n}(\delta)=\Big\{ \mathtt{x}=(x_1,x_2,\dots,x_d)\in \mathbb{T}^{d}: \prod_{i=1}^{d}|T^{n}_{\beta_i}(x_i)-x_i|<\delta\Big\}.\]
By Lemma \ref{big1}, we obtain   a cover of $E_{n}(\delta)$ in the following Lemma \ref{main-lemma}.
\begin{lemma}\label{main-lemma}
With the same notation given in Lemma \ref{big1}, we have 
\begin{equation}\label{mainlemma-equ}
E_n(\delta)\subset \bigcup_{B\in \mathcal{B}}\bigcup_{\mathtt{w}_1\in \Sigma_{\beta_1}^{n},\cdots,\mathtt{w}_d\in\Sigma_{\beta_d}^{n}}   \prod_{i=1}^{d} E_{n,\beta_i}(\mathtt{w}_i,a_{B},b_{B}). 
\end{equation}

\end{lemma}
\begin{proof}
    For each $\mathtt{x}\in E_n(\delta)$, it is direct that 
    \[\big(|T^{n}_{\beta_1}(x_1)-x_1|,\cdots,|T^{n}_{\beta_d}(x_d)-x_d|\big)\in H_{d}(\delta).\]
    Thus for some $B\in \mathcal{B}$, 
    \[a_{B}\leq |T^{n}_{\beta_i}(x_i)-x_i|\leq b_{B}, \text{\ for \ all \ }1\leq i\leq d. \]
    This gives that 
    \begin{align*}
        \mathtt{x}&\in \prod_{i=1}^{d} \big\{x_i\in [0,1): a_{B}\leq |T^{n}_{\beta_i}(x_i)-x_i|\leq b_{B}\big\}\\
        &=\bigcup_{\mathtt{w}_1\in \Sigma_{\beta_1}^n,\cdots, \mathtt{w}_d\in \Sigma_{\beta_d}^n}\prod_{i=1}^{d}E_{n,\beta_i}(\mathtt{w}_i,a_{B},b_{B}).
    \end{align*}
\end{proof}

We now begin to prove Theorem \ref{hypercube}. 
Recall that 
\begin{align*}
 H_d(T,\psi) = \bigg\{\mathtt{x}\in \mathbb{T}^d:\prod_{i=1}^{d}|T_{\beta_i}^n x_i-x_i|\leq \psi(n) \text{\ for \ i.m. \ } n\in \mathbb{N} \bigg\}.
\end{align*}
Let 
\begin{equation*}
     E_n(\psi)= \bigg\{\mathtt{x}\in \mathbb{T}^d:\prod_{i=1}^{d}|T_{\beta_i}^n x_i-x_i|< \psi(n) \bigg\}.
\end{equation*}
Then  $   H_d(T,\psi)=\limsup_{n\to \infty}E_n(\psi)$.

We divided the proof  into two parts: the upper bound and the lower bound, as is typical when calculating the Hausdorff dimension of a set.

As usual, the upper bound is obtained by finding an efficient cover of $ H_d(T,\psi)$.
First, for sufficiently large $n$, we  find an efficient cover of $E_n(\psi)$.
By Lemmas \ref{big1} and \ref{main-lemma}, with $\delta=\psi(n)$ and $n>\log_{\beta_1} 2$, for any $s\in (d-1,d)$,  there exists  a collection  $\mathcal{B}$ of
cubes $B$ such that 
\begin{equation}\label{en-cover}
    E_n(\psi)\subset \bigcup_{B\in \mathcal{B}}\bigcup_{\mathtt{w}_1\in \Sigma_{\beta_1}^{n},\cdots,\mathtt{w}_d\in\Sigma_{\beta_d}^{n}}   \prod_{i=1}^{d} E_{n,\beta_i}(\mathtt{w}_i,a_{B},b_{B}).
\end{equation}
and
\begin{equation}\label{d(B)}
    \sum_{B\in \mathcal{B}}|B|^s\ll ( \psi(n))^{s-d+1}.
\end{equation}
 By Lemma \ref{big}, we obtain that $E_{n,\beta_i}(\mathtt{w}_i,a_{B},b_{B})$ can be covered by $4$ intervals of length $(b_{B}-a_{B})\beta_i^{-n}$  for  any $1\leq i\leq d$. Thus, $\prod_{i=1}^{d} E_{n,\beta_i}(\mathtt{w}_i,a_{B},b_{B})$ can be covered by $4^d$  hyperrectangles, and the side length of these hyperrectangles in the
  direction of the $i$-th axis
 is $(b_{B}-a_{B})\beta^{-n}_i=|B|\beta_i^{-n}$.

We now cover $\prod_{i=1}^{d} E_{n,\beta_i}(\mathtt{w}_i,a_{B},b_{B})$ by balls with diameter equal to $ \beta_d^{-n}$.
Note that for any $1\leq i\leq d$, $\beta_i^{-n}\geq \beta_d^{-n}$, then we can find a collection $\mathcal{C}_n$ of balls with diameter $\beta_d^{-n}|B|$ that cover $\prod_{i=1}^{d} E_{n,\beta_i}(\mathtt{w}_i,a_{B},b_{B})$ with 
\begin{equation*}
    \# \mathcal{C}_n \leq 4^d \prod_{i=1}^{d}\bigg(\frac{\beta_i^{-n}|B|}{\beta_d^{-n}|B|}+1\bigg)=4^d\prod_{i=1}^{d}\bigg(\frac{\beta_i^{-n}}{\beta_d^{-n}}+1\bigg)\leq 8^d \prod_{i=1}^{d}\frac{\beta_i^{-n}}{\beta_d^{-n}}.
\end{equation*}
Hence for any $N>\log_{\beta_1}2$,
\[ H_d(T,\psi)\subseteq \bigcup_{n=N}^{\infty}\bigcup_{B\in \mathcal{B}}\bigcup_{\mathtt{w}_1\in \Sigma_{\beta_1}^{n},\cdots,\mathtt{w}_d\in\Sigma_{\beta_d}^{n}}\bigcup_{\hat{B}\in \mathcal{C}_n}(\hat{B})^s=\bigcup_{n=N}^{\infty}\bigcup_{B\in \mathcal{B}}\bigcup_{\mathtt{w}_1\in \Sigma_{\beta_1}^{n},\cdots,\mathtt{w}_d\in\Sigma_{\beta_d}^{n}}\bigcup_{\hat{B}\in \mathcal{C}_n}(\beta_d^{-n}|B|)^s.\]
  Given $\varepsilon>0$, we choose  
  $N$  large enough so that for any $n\geq \max\{ N, \log_{\beta_1}2\}$ and  any ball $B\in \mathcal{B}$,  $\beta_d^{-n}|B|<\varepsilon$.
Then,  according to the definition of the $s$-dimensional Hausdorff measure, for any $s>0$
\begin{align}\label{Rdimension}
    \mathcal{H}_{\varepsilon}^s( H_d(T,\psi))&\leq \sum_{n=N}^{\infty}\sum_{B\in \mathcal{B}} \Big(\prod_{i=1}^d\#\Sigma^n_{\beta_i}\Big)\cdot\#\mathcal{C}_n \cdot\big(\beta_d^{-n}|B|\big)^s \nonumber\\
    &\leq \sum_{n=N}^{\infty}\sum_{B\in \mathcal{B}} \Big(\prod_{i=1}^d\#\Sigma^n_{\beta_i}\Big)\cdot8^d \prod_{i=1}^{d}\frac{\beta_i^{-n}}{\beta_d^{-n}}\cdot\big(\beta_d^{-n}|B|\big)^s.
\end{align}
On the one hand, by Lemma \ref{cylinder-number},  $\#\Sigma^{n}_{\beta}\leq  \frac{\beta^{n+1}}{\beta-1}$.
This together with (\ref{d(B)}) and (\ref{Rdimension}) implies  there exists $c\in\mathbb{R}$ such that  for any $s\in (d-1,d)$ and  sufficiently large $N$
\begin{align*}
 \mathcal{H}_{\varepsilon}^s( H_d(T,\psi))&\leq \sum_{n=N}^{\infty} \prod_{i=1}^{d} \frac{\beta_i^{n+1}}{\beta_i-1} \bigg( 8^d \prod_{i=1}^{d}\frac{\beta_i^{-n}}{\beta_d^{-n}}\bigg)\beta_d^{-n s}
 c( \psi(n))^{s-d+1}\\
 &\leq 4^d \sum_{n=N}^{\infty} \beta_d^{n(d-s)}c\psi(n)^{s-d+1}\\
 &=4^d c\sum_{n=N}^{\infty} \exp \bigg( n\Big(d\log \beta_d-(d-1)\frac{\log \psi(n)}{n}-s(\log \beta_d-\frac{\log \psi(n)}{n})\Big)\bigg). 
\end{align*}
Therefore, for any 
\[s> d-1+\frac{\log \beta_d}{\log \beta_d+\tau},\]
$\mathcal{H}^s( H_d(T,\psi))=0$, which implies that 
\begin{equation*}
    \dim_{\mathrm{H}}  H_d(T,\psi)\leq d-1+\frac{\log \beta_d}{\log \beta_d+\tau}.
\end{equation*}

Now, we will give the proof of the lower bound  of  $ \dim_{\mathrm{H}}  H_d(T,\psi)$.
\begin{lemma}[\cite{Falconer1990}]\label{cap}
Let $F$ be a subset of $\mathbb{R}^d$ and  $E$ be a subset of the $x_d$-axis. Assume that for all $x \in E$
\begin{equation*}
    \dim_{\mathrm{H}} F\cap L_x \geq t
\end{equation*}
where $L_x$ is the plane parallel to all other axis through the point $(0,\dots,0,x).$
Then 
\begin{equation*}
    \dim_{\mathrm{H}}F\geq t+\dim_{\mathrm{H}}E.
\end{equation*}
\end{lemma}
For any $n\geq 1$, letting  $f_n(x)=x$  and $d=1$ in Theorem \ref{main theorem1}, we have  that 
\begin{equation*}
    \dim_{\mathrm{H}} H_1(T,\psi)=\dim_{\mathrm{H}}\{x\in \mathbb{T}:|T_{\beta_d}^n x-x|\leq \psi(n) \text{\ for\ i.m.\ } n\in \mathbb{N}\}=\frac{\log \beta_d}{\log \beta_d+\tau},
\end{equation*}
where $\tau=\liminf_{n\to \infty}\frac{-\log \psi(n)}{n}$.

For any $x_d\in H_1(T,\psi)$,
\begin{equation*}
    \big([0,1)^{d-1}\times  H_1(T,\psi) \big)\cap L_{x_d}=[0,1)^{d-1},
\end{equation*}
which follows that 
\begin{equation*}
    \dim_{\mathrm{H}}\Big( \big([0,1)^{d-1}\times H_1(T,\psi) \big)\cap L_{x_d} \Big)\geq d-1.
\end{equation*}
By  Lemma \ref{cap}, we arrive at
\begin{equation}\label{d-1}
    \dim_{\mathrm{H}}\Big( [0,1)^{d-1}\times H_1(T,\psi) \Big)\geq d-1+\frac{\log \beta_d}{\log \beta_d+\tau}.
\end{equation}
Recalling  the definition of $H_d(T,\psi)$, we  obtain that 
\begin{equation}\label{d}
     [0,1)^{d-1}\times H_1(T,\psi) \subset H_d(T,\psi).
\end{equation}
Therefore, by \eqref{d-1} and \eqref{d}, we have 
\begin{equation*}
    \dim_{\mathrm{H}} H_d(T,\psi)\geq d-1+\frac{\log \beta_d}{\log \beta_d+\tau}.
\end{equation*}

\section{Proof of Theorem \ref{main theorem1} and  Corollary \ref{main_corollary}}\label{main section}
We first provide the upper and lower bounds of $\dim_{\mathrm{H}}W(T,\Psi,{f_n})$ separately in Sections \ref{upper bound} and \ref{lower}.
   Then we will give the proof of   Corollary \ref{main_corollary} in Section \ref{corollary}.

Suppose  $T$ is a diagonal matrix,  there exist $\beta_1, \dots,\beta_d \in \mathbb{R}^{+}$ such that  $T=\text{diag}(\beta_1,\beta_2,\dots,\beta_d)$. Then 
\begin{equation*}
    W(T,\Psi, \{f_n\})=\big\{\mathtt{x}\in \mathbb{T}^d: |T_{\beta_i}^n(x_i)-f^{(i)}_n(x_i)| \leq \psi_i(n) (1\leq i\leq d) \text{\ for \  i.m.\ }  n\in \mathbb{N}\big\},
\end{equation*}
where $T_{\beta_i}$ is the standard $\beta$-transformation with $\beta=\beta_i$.

\subsection{The upper bound part}\label{upper bound}
Recall that $\mathcal{C}(\Psi)$ is the set of  accumulation  point  $\mathbf{t}=(t_1,\dots,t_d)$ of  the sequence $ \{ \frac{-\log \psi_1(n)}{n},\dots, \frac{-\log \psi_d(n)}{n}\}_{n\geq 1}$. 

We first assume that  $\mathcal{C}(\Psi)$ contains only one point, then for any $1\leq i\leq d$,
\begin{equation*}
    \lim_{n\to \infty}\frac{-\log \psi_i(n)}{n}=t_i.
\end{equation*}
For any  $\mathtt{w}_i\in \Sigma_{\beta_i}^n$, $1\leq i \leq d$, let 
\begin{equation*}
    J_{n,\beta_i}(\mathtt{w}_i)=\big\{x_i \in I_{n,\beta_i}(\mathtt{w}_i):|T^n_{\beta}x_i-f^{(i)}_n(x_i)|<\psi_i(n)\big\}.
\end{equation*}
Then  we have 
\begin{equation}\label{cover}
    W(T, \Psi, \{f_n\})=\limsup_{n\to \infty}\bigcup_{\mathtt{w}_1\in \Sigma_{\beta_1}^n} \cdots \bigcup_{\mathtt{w}_d\in \Sigma_{\beta_d}^n} J_{n,\beta_1}(\mathtt{w}_1)\times J_{n,\beta_2}(\mathtt{w}_2)\times\cdots\times J_{n,\beta_d}(\mathtt{w}_d).
\end{equation}
For any $x_i, x_i'\in J_{n,\beta_i}(\mathtt{w}_i)$, we have 
\begin{align*}
    2\psi_i(n)&\geq |T^n_{\beta_i}x_{i}-f^{(i)}_n(x_{i})|+|T^n_{\beta_i}x'_{i}-f^{(i)}_n(x'_{i})| \nonumber \\ \nonumber
    &\geq |T^n_{\beta_i}x_{i}-T^n_{\beta_i}x'_{i}|-|f^{(i)}_n(x_{i})-f^{(i)}_n(x'_{i})|\\
    &\geq (\beta_i^n-c)|x_{i}-x'_{i}|.
\end{align*}
Therefore, for  sufficiently  large $n$,  we  obtain that  for all $1\leq i \leq d$, 
\begin{equation}\label{Jn}
    |J_{n,\beta_i}(\mathtt{w}_i)|\leq 4\psi_i(n)\beta_i^{-n}.
\end{equation}
By \eqref{cover},  for any $N\geq 1$, 
\begin{equation}\label{JN1}
        W(T, \Psi, \{f_n\})\subset  \bigcup_{n=N}^{\infty}\mathcal{D}_n,
\end{equation}
where 
\[\mathcal{D}_n= \bigcup_{\mathtt{w}_1\in \Sigma_{\beta_1}^n} \cdots \bigcup_{\mathtt{w}_d\in \Sigma_{\beta_d}^n} J_{n,\beta_1}(\mathtt{w}_1)\times J_{n,\beta_2}(\mathtt{w}_2)\times\cdots\times J_{n,\beta_d}(\mathtt{w}_d).\]
The upper bound of $\dim_{\text{H}}W(T,\Psi,\{f_n\})$ can be proven using the \eqref{Jn} and \eqref{JN1} in a manner similar to \cite[Proposition 4]{li2022}.

\subsection{The lower bound part} \label{lower}
In this section, we will prove that 
\begin{equation}\label{lower_bound}
     \dim_{\text{H}} W(T, \Psi, \{f_n\})\geq  \sup_{\mathbf{t}\in \mathcal{C}(\Psi)} \min_{1\leq i \leq d}\{ \lambda_i(\mathbf{t})\}.
\end{equation}
Now we will construct a subset of $ J_{n,\beta_i}(\mathtt{w}_i)$ for  all full  cylinder $I_{n,\beta_i}(\mathtt{w}_i)$.
Denote 
\begin{equation}\label{full_admissible}
    \bar{\Sigma}_{\beta}^n:=\{u\in \Sigma_{\beta}^n: I_{n,\beta}(u)  \text{ is \ a\ full \ cylinder }\}.
\end{equation}
For any full cylinder $I_{n,\beta_i}(\mathtt{w}_i)$,  by Lemma \ref{wu}, there exists a point $x_{n,\mathtt{w}_i}\in I_{n,\beta_i}(\mathtt{w}_i)$ such that 
\begin{equation*}
    |T^n_{\beta_i}x_{n,\mathtt{w}_i}-f^{(i)}_n(x_{n,\mathtt{w}_i})|< \frac{1}{2}\psi_i(n).  
\end{equation*}
Then for any  $n\geq \frac{\log c}{\log \beta_i}$, $x \in I_{n,\beta_i}(\mathtt{w}_i)$ and  $|x-x_{n,\mathtt{w}_i}|\leq \frac{1}{4}\beta_i^{-n}\psi_i(n)$,
\begin{align*}
    | T^n_{\beta_i}x-f^{(i)}_n(x)| &\leq  | T^n_{\beta_1}x-T^n_{\beta_i}x_{n,\mathtt{w}_i}|+ |T^n_{\beta_i}x_{n,\mathtt{w}_i}-f^{(i)}_n(x_{n,\mathtt{w}_i})|+|f^{(i)}_n(x_{n,\mathtt{w}_i})-f^{(i)}_n(x)|\\
    &\leq (\beta_i^n+c)|x-x_{n,\mathtt{w}_i}|+\frac{1}{2}\psi_i(n)\\
    &\leq  \psi_i(n).
\end{align*}
Let $a$ and $b$ be the left and right endpoints of the cylinder $I_{n,\beta_i}(\mathtt{w}_i)$.
If 
\[\max\big\{|x_{n,\mathtt{w}_i}-a|, |x_{n,\mathtt{w}_i}-b|\big\}\leq \frac{1}{4}\beta_i^{-n}\psi_i(n),\]
then we can choose a point $x_{n,\mathtt{w}_i}^{*}\in I_{n,\beta_i}(\mathtt{w}_i) $ such that 
\[|x_{n,\mathtt{w}_i}^{*}-x_{n,\mathtt{w}_i}|\leq  \frac{1}{8}\beta_i^{-n}\psi_i(n) \ \text{and} \  B(x_{n,\mathtt{w}_i}^{*}, \frac{1}{8}\beta_i^{-n}\psi_i(n))\subset J_{n,\beta_i}(\mathtt{w}_i).\]
If 
$\min \big\{|x_{n,\mathtt{w}_i}-a|, |x_{n,\mathtt{w}_i}-b|\big\}> \frac{1}{4}\beta_i^{-n}\psi_i(n),$  let $x_{n,\mathtt{w}_i}^{*}:=x_{n,\mathtt{w}_i}$.
Therefore,
\begin{align*}                   
     W(T, \Psi, \{f_n\})&\supset \limsup_{n\to \infty} \bigcup_{\mathtt{w}_1\in \bar{\Sigma}_{\beta_1}^n} \bigcup_{\mathtt{w}_2\in \bar{\Sigma}_{\beta_2}^n} \cdots \bigcup_{\mathtt{w}_d\in \bar{\Sigma}_{\beta_d}^n}  B(x_{n,w_1}^{*}, \frac{1}{8}\beta_1^{-n}\psi_1(n))\\ 
      &\hspace{7em} \times  B(x_{n,\mathtt{w}_2}^{*}, \frac{1}{8}\beta_2^{-n}\psi_2(n))\times\cdots \times  B(x_{n,\mathtt{w}_d}^{*}, \frac{1}{ 8}\beta_d^{-n}\psi_d(n)).
\end{align*}
By  Lemmas \ref{cylinder-length} and \ref{wangcylinder}, we know that the distance between any consecutive  $n$-th  full cylinder along the $i$-th axis is less than $(n+3)\beta_i^n$, which implies that 
\begin{equation}\label{cover_T}
  \mathbb{T}\subset  \big\{ B(x_{n,\mathtt{w}_i}, (n+3)\beta_i^{-n}): \mathtt{w}_i\in \bar{\Sigma}_{\beta_i}^n\big\}.
\end{equation}
With  inequalities \eqref{lower_bound}-\eqref{cover_T} and  Lemma \ref{masstranference}, the lower bound of $\dim_{\text{H}}W(T,\Psi,\{f_n\})$ can be proved in the similar
way as \cite[Proposition 5]{li2022}.
\subsection{The proof of Corollary \ref{main_corollary}}\label{corollary}
Recall that  $\mathcal{C}(\Psi)$ is  the set of  accumulation  points  $\mathbf{t}=(t_1,\dots,t_d)$ of  the sequence $ \Big\{ \frac{-\log \psi_1(n)}{n},\dots, \frac{-\log \psi_d(n)}{n}\Big\}_{n\geq 1}$.
If $\psi_1=\psi_2=\dots=\psi_d=\psi$,  we use $\mathcal{C}(\psi)$ instead of  $\mathcal{C}(\Psi)$.
Thus, for any  $\mathbf{t}\in \mathcal{C}(\psi)$, $\mathbf{t}=\{t,t, \dots,t\}$.

If $\tau=\liminf_{n\to \infty}\frac{-\log \psi(n)}{n}<\infty$, then $ \mathcal{C}(\psi)$ is bounded.  
By  Theorem \ref{main theorem1},  
\begin{equation*}
    \dim_{\mathrm{H}}W(T, \psi, \{f_n\})=\sup_{\mathbf{t}\in \mathcal{C}(\Psi)}\min_{1\leq i\leq d} \{\lambda_i(\mathbf{t})\}=\sup_{\mathbf{t}\in \mathcal{C}(\Psi)}\min_{1\leq i\leq d} \{\lambda_i(t)\},
\end{equation*}
 for any $1\leq i\leq d$,
\begin{align}\label{lamada}
    \lambda_i(t)
    &=\sum_{j: \beta_j >\beta_i e^t}1+\sum_{j: \beta_j\leq \beta_i}\frac{ \log \beta_i}{\log \beta_i +t}+\sum_{j:\beta_i e^t\geq\beta_j>\beta_i}\frac{\log \beta_j}{\log \beta_i+t}.
\end{align}
Note that $ \lambda_i(t)$ is a decreasing function in $t$, then 
\[\sup_{\mathbf{t}\in \mathcal{C}(\psi)}\min_{1\leq i\leq d} \{\lambda_i(t)\}=\min_{1\leq i\leq d} \Bigg\{\frac{i\log \beta_i-\sum_{j:\beta_j>\beta_i e^{\tau}}(\log \beta_j-\log \beta_i -\tau)+\sum_{j>i}\beta_j}{\tau+\log \beta_i}\Bigg\}.\]

If $\tau= \infty$, for any $M>0$, let $\psi_M: \mathbb{R}^{+}\to \mathbb{R}^{+}:x\to e^{-x M}$. Then $\mathcal{C}(\psi_M)=\{M\}$.
We can choose  $M$ sufficiently large such that 
\begin{equation*} 
   W(T, \psi, \{f_n\})\subset W(T, \psi_M, \{f_n\}),
\end{equation*}
which yields that 
\begin{equation}
    0\leq  \dim_{\mathrm{H}} W(T,\psi, \{f_n\})\leq  \dim_{\mathrm{H}}W(T, \psi_M, \{f_n\})\leq \min_{1\leq i\leq d}\lambda_i(M),
\end{equation}
where 
\[\lambda_i(M)=\sum_{j: \beta_j >\beta_i e^M}1+\sum_{j: \beta_j\leq \beta_i}\frac{ \log \beta_i}{\log \beta_i +M}+\sum_{j:\beta_i e^M\geq\beta_j>\beta_i}\frac{\log \beta_j}{\log \beta_i+M}.\]
Then we  obtain for any $1\leq i\leq d$,
$ \lim_{M\to \infty}\lambda_i(M)=0.$
Therefore, 
\begin{equation}
    \dim_{\mathrm{H}} W(T,\psi, \{f_n\})=0.
\end{equation}

\section{Proof of Theorem \ref{diagonalizable}}\label{diagona}

Recall that 
\[ W(T,\psi, \{f_n\})=\big\{\mathtt{x}\in \mathbb{T}^d: T^n(\mathtt{x}) \in B(f_n(\mathtt{x}), \psi(n)) \text{\ for \  i.\ m.\ }  n\in \mathbb{N}\big\},\]
and  $T$ is diagonalizable over $\mathbb{Z}$ with the  eigenvalues  $1<\beta_1\leq \beta_2\leq \dots\leq \beta_d$.
Then, there exists a diagonal integer matrix $D$ and an invertible mapping $\phi$ satisfying $\phi \circ T=D \circ \phi$. Note that $\phi$ is a bi-Lipschitz map,  there exist two constants $0<c_1\leq c_2<\infty$ such that for any $x, y\in \mathbb{T}^d$,
 \begin{equation}\label{lip_phi}
     c_1|x-y|\leq |\phi(x)-\phi(y)|\leq c_2|x-y|.
 \end{equation}
 Let 
 \[S:=\big\{\mathtt{x}\in \mathbb{T}^d: T^n(\mathtt{x})\in B\big(f_n(\phi(\mathtt{x})), \frac{c_2}{c_1}\psi(n)\big)  \text{ \  for \  i.m.} \ n\in \mathbb{N} \big\}\]
 and 
 \[S_j=\big\{\mathtt{x}\in \mathbb{T}^d: D^n(\mathtt{x})\in B(\phi(f_n(\mathtt{x})), \frac{c_2^2}{c_j}\psi(n)) \ \text{ i.m.}\  n\in \mathbb{N}\big\} \text{ \ for\ } j=1,2.\]
Since $\{f_n\}$ is a uniform Lipschitz sequence and $\phi$ is a  bi-Lipschitz map, $\{\phi\circ f_n \}$  is also a uniform Lipschitz sequence. 
 By  Corollary \ref{main_corollary}, we can obtain that   
\begin{equation*}
    \dim_{\mathrm{H}}S_1= \dim_{\mathrm{H}}S_2=\min_{1\leq i\leq d} \Bigg\{\frac{i\log \beta_i-\sum_{j:\beta_j>\beta_i e^{\tau}}(\log \beta_j-\log \beta_i -\tau)+\sum_{j>i}\beta_j}{\tau+\log \beta_i}\Bigg\},
\end{equation*}
  which implies that
  the Hausdorff dimension of $S_j$  is independent of the uniform Lipschitz sequence $\{f_n\}$.
 We claim that 
 \begin{align}\label{claim}
    S_2\subset \phi(S)\subset S_1,
 \end{align}
 then
  $\dim_{\mathrm{H}}S_1= \dim_{\mathrm{H}}S_2=\dim_{\mathrm{H}}\phi(S)$.
 By   Lemma \ref{lipschitz function} and the definition of $\phi$, we can obtain  $\dim_{\mathrm{H}}\phi(S)=\dim_{\mathrm{H}}S$.
 Therefore,   we  need only to prove the claim \eqref{claim}.
 
For any $y\in \phi(S)$, there exists a point  $x_0\in S$ such that $y=\phi(x_0)$.
Due to the fact that $ \phi \circ T=D \circ \phi$ and (\ref{lip_phi}),
\[|D^n\phi(x_0)-\phi(f_n(\phi(x_0)))|=|\phi(T^n x_0)-\phi(f_n(\phi(x_0)))|\leq c_2|T^n x_0-f_n(\phi(x_0))|.\]
 Note that $x_0\in S$, then  $|T^n x_0-f_n(\phi(x_0))|<\frac{c_2}{c_1}\psi(n)$ for  infinity many $n\in \mathbb{N}$, which implies that 
 \begin{equation*}
    |D^n(y)-\phi(f_n(y))|= |D^n\phi(x_0)-\phi(f_n(\phi(x_0)))|\leq \frac{c_2^2}{c_1}\psi(n).
 \end{equation*}
 Thus, $y=\phi(x_0)\in S_1$, which implies that $\phi(S)\subset S_1$.
 
 For any  $z\in S_2$, there exists $x_1\in \mathbb{T}^d$ such that $z=\phi(x_1)$, which yields  that 
 \begin{equation*}
  |D^n\phi(x_1)-\phi(f_n(\phi(x_1)))|= |\phi(T^n x_1)-\phi(f_n(\phi(x_1)))| \leq c_2\psi(n),
 \end{equation*}
 for  infinity many $n\in \mathbb{N}$.
This together with  
  \begin{equation*}
    |\phi(T^n x_1)-\phi(f_n(\phi(x_1)))|\geq c_1|T^n x_1-f_n(\phi(x_1))|
 \end{equation*}
implies that $x_1\in S$ and $z=\phi(x_1)\in \phi(S)$,  then $ S_2\subset \phi(S)$.
  
 Thus, we complete the proof of the claim.

\subsection*{Acknowledgements}
 This work was supported partially by Guangdong Natural Science Foundation 2023A1515010691 and China Scholarship Council.
 
\bibliographystyle{amsplain}
\bibliography{main}

\end{document}